\documentclass{amsart}
\usepackage[a4paper, margin=3cm]{geometry}
\usepackage{amsthm}
\usepackage{amsmath,amssymb}
\usepackage{mathrsfs}
\usepackage{url}
\newtheorem{theorem}{Theorem}[section]
\newtheorem*{theorem*}{Theorem}
\newtheorem{lemma}[theorem]{Lemma}
\newtheorem*{lemma*}{Lemma}
\newtheorem{proposition}[theorem]{Proposition}
\newtheorem*{proposition*}{Proposition}
\newtheorem{corollary}[theorem]{Corollary}
\newtheorem*{corollary*}{Corollary}

\newtheorem*{claim*}{Claim}

\newtheorem*{fact*}{Fact}
\theoremstyle{definition}

\newtheorem*{definiton*}{Definition}

\newtheorem*{example*}{Example}
\newtheorem{remark}[theorem]{Remark}
\newtheorem*{remark*}{Remark}

\newtheorem*{question*}{Question}

\newtheorem*{conjecture*}{Conjecture}

\address{graduate school of mathematical science, the university of tokyo, 3-8-1 komaba, meguroku, tokyo 153-8914, japan.}
\email{kohara@ms.u-tokyo.ac.jp}
\begin{document}
\title{hecke algebras for tame supercuspidal types}
\author{Kazuma Ohara}
\date{}
\maketitle
\begin{abstract}
Let $F$ be a non-archimedean local field of residue characteristic $p\neq 2$. Let $G$ be a connected reductive group over $F$ that splits over a tamely ramified extension of $F$. In \cite{Yu}, Yu constructed types which are called \emph{tame supercuspidal types} and conjectured that Hecke algebras associated with these types are isomorphic to Hecke algebras associated with depth-zero types of some twisted Levi subgroups of $G$. In this paper, we prove this conjecture. We also prove that the Hecke algebra associated with a \emph{regular supercuspidal type} is isomorphic to the group algebra of a certain abelian group. 
\end{abstract}
\section{Introduction}
Let $F$ be a non-archimedean local field of residue characteristic $p\neq 2$ and $G$ be a connected reductive group over $F$ that splits over a tamely ramified extension of $F$. As explained in \cite{Ber}, the category $\mathcal{R}(G(F))$ of smooth complex representations of $G(F)$ is decomposed into a product $\prod_{[M, \sigma]_{G}} \mathcal{R}^{[M, \sigma]_G}(G(F))$ of full subcategories $\mathcal{R}^{[M, \sigma]_G}(G(F))$, called \emph{Bernstein blocks}. Bernstein blocks are parametrized by inertial equivalence classes $[M, \sigma]_{G}$ of cuspidal pairs. Each block $\mathcal{R}^{[M,\sigma]_G}(G(F))$ is equivalent to the category of modules over an algebra if $[M,\sigma]_G$ has an associated type as explained below. Let $K$ be a compact open subgroup of $G(F)$, $(\rho, W)$ be an irreducible representation of $K$, and $\mathfrak{s}$ be an inertial equivalence class of a cuspidal pair. We say that $(K, \rho)$ is an $\mathfrak{s}$-type if $\mathcal{R}^{\mathfrak{s}}(G(F))$ is precisely the full subcategory of $\mathcal{R}(G(F))$ consisting of smooth representations which are generated by their $\rho$-isotypic components. In this case, $\mathcal{R}^{\mathfrak{s}}(G(F))$ is equivalent to the category of modules over the Hecke algebra $\mathcal{H}(G, \rho)$ associated with $(K, \rho)$ \cite[Theorem~4.3]{BK2}. Therefore, to construct types and determine the structure of Hecke algebras associated with the types are essential to understand the category $\mathcal{R}(G(F))$. 

In \cite{MP1} and \cite{MP2}, Moy and Prasad defined the notion of \emph{depth} of types and constructed \emph{depth-zero types}.
Hecke algebras associated with depth-zero types were calculated in \cite{Mo}. In \cite{Mo}, Morris gave the generators and relations for Hecke algebras associated with depth-zero types \cite[Theorem~7.12]{Mo}.

In \cite{Yu}, Yu constructed types of general depth which are called \emph{tame supercuspidal types}. His construction starts with a tuple $(\overrightarrow{G}, y, \overrightarrow{r}, {^\circ\!\rho_{-1}}, \overrightarrow{\phi})$, out of which it produces a sequence of types $({^\circ\!K^i}, {^\circ\!\rho_{i}})$ in $G^i(F)$, where $\overrightarrow{G}=\left(G^0\subsetneq G^1 \subsetneq\ldots \subsetneq G^d=G\right)$ is a sequence of twisted Levi subgroups of $G$. Yu conjectured that the Hecke algebras associated with $({^\circ\!K^i}, {^\circ\!\rho_{i}})$ are all isomorphic \cite[Conjecture~0.2]{Yu}. In particular, Hecke algebras associated with these types are isomorphic to Hecke algebras associated with depth-zero types, which are studied in \cite{Mo} as explained above.
In \cite{AM}, Adler and Mishra proved \cite[conjecture~0.2]{Yu} under some conditions \cite[Corollary~6.4]{AM}. However, this result covers only the cases that Hecke algebras are commutative.
In this paper, we prove \cite[Conjecture~0.2]{Yu} without any assumptions. This is the first topic of this paper. 

The second topic of this paper is on \emph{regular supercuspidal types}. In \cite{Kal}, Kaletha defined and constructed a large class of supercuspidal representations which he calls \emph{regular}. Kaletha's construction starts with a regular tame elliptic pair $(S, \theta)$, where $S$ is a tame elliptic maximal torus, and $\theta$ is a character of $S(F)$ which satisfy some conditions \cite[Definition~3.7.5]{Kal}.  As explained in the paragraph following \cite[Definition~3.7.3]{Kal}, ``most'' supercuspidal representations are regular when $p$ is not too small. In this paper, we define and construct \emph{regular supercuspidal types}, which are constructed by the same data $(S, \theta)$ as Kaletha's construction of regular supercuspidal representations. The regular supercuspidal representation constructed by $(S,\theta)$ contains the regular supercuspidal type constructed by the same $(S, \theta)$. We prove that the Hecke algebra associated with the type constructed by $(S, \theta)$ is isomorphic to the group algebra of the quotient of $S(F)$ by the unique maximal compact subgroup. This result is proved independent of \cite[Conjecture~0.2]{Yu} and the work of \cite{Mo}.

We sketch the outline of this paper. In Section \ref{secyu}, we review Yu's construction of supercuspidal representations briefly. In Section \ref{sechecke}, we prove \cite[Conjecture~0.2]{Yu}. In Section \ref{secreg}, we review Kaletha's construction of regular supercuspidal representations, define regular supercuspidal types, and determine the structure of Hecke algebras associated with regular supercuspidal types.
\subsection*{Acknowledgment}
I am deeply grateful to my supervisor Noriyuki Abe for his enormous support and helpful advice. He checked the draft and gave me useful comments. I would also like to thank Jessica Fintzen. She checked the previous draft and gave me a lot of comments. I am supported by the FMSP program at Graduate School of Mathematical Sciences, the University of Tokyo.
\section{Notation and assumptions}
Let $F$ be a non-archimedean local field of residue characteristic $p$, $k_F$ be its residue field, and $G$ be a connected reductive group over $F$ that splits over a tamely ramified extension of $F$.
We denote by $Z(G)$ the center of $G$ and by $G_{\text{der}}$ the derived subgroup of $G$.
We assume that $p$ is an odd prime.

We denote by $\mathcal{B}(G, F)$ the enlarged Bruhat--Tits building of $G$ over $F$.
If $T$ is a maximal, maximally split torus of $G_E := G \times_{F} E$ for a field extension $E$ over $F$, then $\mathcal{A}(T, E)$ denotes the apartment of $T$ inside the Bruhat--Tits building $\mathcal{B}(G_E, E)$ of $G_E$ over $E$.
For any $y\in \mathcal{B}(G,F)$, we denote by $[y]$ the projection of $y$ on the reduced building and by $G(F)_y$ (resp.\,$G(F)_{[y]}$) the subgroup of $G(F)$ fixing $y$ (resp.\,$[y]$).
For $y\in \mathcal{B}(G, F)$ and $r\in \widetilde{\mathbb{R}}_{\ge 0}=\mathbb{R}_{\ge 0}\cup \{r+\mid r\in \mathbb{R}_{\ge 0}\}$, we write $G(F)_{y, r}$ for the Moy--Prasad filtration subgroup of $G(F)$ of depth $r$ \cite{MP1, MP2}.

Suppose that $K$ is a subgroup of $G(F)$ and $g\in G(F)$. We denote $gKg^{-1}$ by $^gK$. If $\rho$ is a smooth representation of $K$, $^g\!\rho$ denotes the representation $x\mapsto \rho(g^{-1}xg)$ of $^gK$. If $\text{Hom}_{K\cap ^gK}(^g\!\rho, \rho)$ is non-zero, we say $g$ intertwines $\rho$.
\section{Review of Yu's construction}
\label{secyu}
In this section, we recall Yu's construction of supercuspidal representations and supercuspidal types of $G(F)$ \cite{Yu}.

An input for Yu's construction of supercuspidal representations of $G(F)$ is a tuple $(\overrightarrow{G}, y, \overrightarrow{r}, \rho_{-1}, \overrightarrow{\phi})$ where
\begin{description}
\item[\bf{D1}]
$\overrightarrow{G}=\left(G^0\subsetneq G^1 \subsetneq\ldots \subsetneq G^d=G\right)$ is a sequence of twisted Levi subgroups of $G$ that split over a tamely ramified extension of $F$, i.\,e., there exists a tamely ramified extension $E$ of $F$ such that $G^i_{E}$ is split for $0\le i\le d$, and $\left(G^0_{E}\subsetneq G^1_{E} \subsetneq\ldots \subsetneq G^d_{E}=G_{E}\right)$ is a split Levi sequence in $G_{E}$ in the sense of \cite[Section~1]{Yu}; we assume that $Z(G^0)/Z(G)$ is anisotropic;
\item[\bf{D2}]
$y$ is a point in $\mathcal{B}(G^0, F)\cap \mathcal{A}(T, E)$ whose projection on the reduced building of $G^0(F)$ is a vertex, where $T$ is a maximal torus of $G^0$ (hence of $G^i$) whose splitting field $E$ is a tamely ramified extension of $F$; we denote by $\Phi(G^{i}, T, E)$ the corresponding root system of $G^i$ for $0\le i \le d$;
\item[\bf{D3}]
$\overrightarrow{r}=\left(r_0, \ldots , r_d\right)$ is a sequence of real numbers satisfying
\begin{align*}
\begin{cases}
0<r_0<r_1<\cdots <r_{d-1}\le r_{d} & (d>0),\\
0\le r_0 & (d=0);
\end{cases}
\end{align*}
\item[\bf{D4}]
$\rho_{-1}$ is an irreducible representation of $G^0(F)_{[y]}$ such that $\rho_{-1}\restriction_{G^0(F)_{y, 0}}$ is the inflation of a cuspidal representation of $G^0(F)_{y,0}/ G^0(F)_{y, 0+}$;
\item[\bf{D5}]
$\overrightarrow{\phi}=\left(\phi_0, \ldots , \phi_d\right)$ is a sequence of characters, where $\phi_i$ is a character of $G^i(F)$; we assume that $\phi_i$ is trivial on $G^i(F)_{y, r_i+}$ but non-trivial on $G^i(F)_{y, r_i}$ for $0\le i\le d-1$. If $r_{d-1}<r_d$, we assume that $\phi_d$ is trivial on $G^d(F)_{y, r_d+}$ but non-trivial on $G^d(F)_{y, r_d}$, otherwise we assume that $\phi_d=1$. Moreover, we assume that $\phi_i$ is $G^{i+1}$-generic of depth $r_i$ relative to $y$ in the sense of \cite[Section.~9]{Yu} for $0\le i\le d-1$. 
\end{description}

Using the datum, we define
\[
\left\{
\begin{aligned}
K^i&=G^0(F)_{[y]}G^{1}(F)_{y, r_0/2}\cdots G^i(F)_{y, r_{i-1}/2},\\
K^i_{+}&=G^0(F)_{y, 0+}G^1(F)_{y, (r_0/2)+}\cdots G^i(F)_{y, (r_{i-1}/2)+}
\end{aligned}
\right.
\]
for $0\le i\le d$.
We also define subgroups $J^i, J^i_{+}$ of $G$ for $1\le i\le d$ as follows. 
For $\alpha\in \Phi(G, T, E)$, let $U_{\alpha}=U_{T, \alpha}$ denote the root subgroup of $G$ corresponding to $\alpha$. We set $U_{0}=T$. For $x\in \mathcal{B}(G, F)$, $\alpha\in \Phi(G, T, E)\cup\{0\}$, and $r\in \widetilde{\mathbb{R}}_{\ge 0}$, let $U_{\alpha}(E)_{x, r}$ denote the Moy--Prasad filtration subgroup of $U_{\alpha}(E)$ of depth $r$ \cite{MP1, MP2}.
We define
\[
\left\{
\begin{aligned}
J^i&= G(F)\cap \langle U_{\alpha}(E)_{y, r_{i-1}}, U_{\beta}(E)_{y, r_{i-1}/2} \mid \alpha\in \Phi(G^{i-1}, T, E)\cup\{0\}, \beta\in \Phi(G^i, T, E)\backslash \Phi(G^{i-1}, T, E) \rangle,\\
J^i_{+}&= G(F)\cap \langle U_{\alpha}(E)_{y, r_{i-1}}, U_{\beta}(E)_{y, (r_{i-1}/2)+} \mid \alpha\in \Phi(G^{i-1}, T, E)\cup\{0\}, \beta\in \Phi(G^i, T, E)\backslash \Phi(G^{i-1}, T, E) \rangle
\end{aligned}
\right.
\]
for $1\le i\le d$.
As explained in \cite[Section~1]{Yu}, $J^i$ and $J^i_{+}$ are independent of the choice of a maximal torus $T$ of $G^0$ so that $T$ splits over a tamely ramified extension $E$ of $F$ and $y\in \mathcal{A}(T, E)$.

Yu constructed irreducible representations $\rho_i$ and $\rho'_i$ of $K^i$ for $0\le i\le d$ inductively. First, we put $\rho'_0=\rho_{-1}, \rho_0=\rho'_0\otimes \phi_0$.

Suppose that $\rho_{i-1}$ and $\rho_{i-1}'$ are already constructed, and $\rho'_{i-1}\restriction_{G^{i-1}(G)_{y, r_{i-1}}}$ is $1$-isotypic. In \cite[Section~11]{Yu}, Yu defined  a representation $\phi'_{i-1}$ of $K^i$ using the theory of Weil representation. This representation only depends on $\phi_{i-1}$. If $r_{i-1}<r_i$, $\phi'_{i-1}\restriction_{G^i(F)_{y, r_i}}$ is $1$-isotypic. Let $\text{inf}\left(\rho'_{i-1}\right)$ be the inflation of $\rho'_{i-1}$ via the map $K^i=K^{i-1}J^i\to K^{i-1}J^i/J^i\simeq K^{i-1}/G^{i-1}(G)_{y, r_{i-1}}$. Now we define $\rho'_i=\text{inf}\left(\rho'_{i-1}\right)\otimes \phi'_{i-1}$, which is trivial on $G^i(F)_{y, r_i}$ if $r_{i-1}<r_i$. Finally, we define $\rho_i=\rho'_i\otimes \phi_{i}$.

We explain the construction of $\phi'_{i-1}$. Let $\langle \cdot, \cdot \rangle_{i}$ be a pairing on $J^{i}/J^{i}_{+}$ defined by $\langle a, b \rangle_i =\hat{\phi}_{i-1}(aba^{-1}b^{-1})$. Here, $\hat{\phi}_{i-1}$ denotes an extension of $\phi_{i-1}\restriction_{K^0G^{i-1}(F)_{y, 0}}$ to $K^0 G^{i-1}(F)_{y, 0} G(F)_{y, (r_{i-1}/2)+}$ defined in \cite[Section~4]{Yu}.
The pairing is well-defined because by \cite[Proposition~6.4.44]{BT}, $[J^{i}, J^{i}]$ is contained in $J^{i}_{+}$.
Note that since the order of every element in $J^i/J^{i}_{+}$ divides $p$, we can regard $J^i/J^i_{+}$ as a $\mathbb{F}_p$-vector space.
By \cite[Lemma~11.1]{Yu}, this pairing is non-degenerate on $J^i/J^i_+$.
In addition, by the construction of $\hat{\phi}_{i-1}$, for $j\in J^{i}_{+}$, $j^{p}$ is contained in $\text{Ker}(\hat{\phi}_{i-1})$. 
Therefore, the order of every element in $\hat{\phi}_{i-1}(J^i_{+})$ divides $p$, and since $\hat{\phi}_{i-1}(J^{i}_{+})$ is a non-trivial subgroup of $\mathbb{C}^{\times}$, this implies that $\hat{\phi}_{i-1}(J^{i}_{+})$ is isomorphic to $\mathbb{F}_{p}$.
Hence we can regard $\langle \cdot, \cdot \rangle_{i}$ as a non-degenerate $\mathbb{F}_p$-valued pairing on $J^{i}/J^{i}_{+}$ and $J^{i}/J^{i}_{+}$ as a symplectic space over $\mathbb{F}_p$. For a symplectic space $\left(V, \langle , \rangle\right)$ over $\mathbb{F}_p$, we define the Heisenberg group $V^{\#}$ of $V$ to be the set $V\times \mathbb{F}_{p}$ with the group law $(v, a)(w, b)=(v+w, a+b+\frac{1}{2}\langle v, w \rangle)$.
Yu constructed a canonical isomorphism $j\colon J^{i}/\left(J^{i}_{+}\cap \text{Ker}(\hat{\phi}_{i-1})\right)\to (J^i/J^i_{+})\times J^i_{+}/\left(J^{i}_{+} \cap \text{Ker}(\hat{\phi}_{i-1})\right)\simeq(J^i/J^i_+)^{\#}$ in \cite[Proposition~11.4]{Yu}. Combining this isomorphism and the map $K^{i-1}\to \text{Sp}\left(J^{i}/J^{i}_{+}\right)$ induced by the conjugation, we define $K^{i-1}\ltimes J^{i}\to K^{i-1}\ltimes \left(J^{i}/\left(J^{i}_{+}\cap \text{Ker}(\hat{\phi}_{i-1})\right)\right)\to \text{Sp}\left(J^{i}/J^{i}_{+}\right)\ltimes (J^{i}/J^{i}_{+})^{\#}$. Combining this map and the Weil representation of $\text{Sp}\left(J^{i}/J^{i}_{+}\right)\ltimes (J^{i}/J^{i}_{+})^{\#}$ associated with the central character $\hat{\phi}_{i-1}$, we construct a representation $\tilde{\phi}_{i-1}$ of $K^{i-1}\ltimes J^{i}$. Let $\text{inf}(\phi_{i-1})$ be the inflation of $\phi_{i-1}$ via the map $K^{i-1} \ltimes J^{i} \to K^{i-1}$, then $\text{inf}(\phi_{i-1})\otimes \tilde{\phi}_{i-1}$ factors through the map $K^{i-1}\ltimes J^i\to K^{i-1}J^{i}=K^{i}$. We define $\phi'_{i-1}$ be the representation of $K^i$ whose inflation to $K^{i-1}\ltimes J^i$ is $\text{inf}(\phi_{i-1})\otimes \tilde{\phi}_{i-1}$. 

For open subgroups $K_1, K_2$ of $G^i(F)$ and a representation $\rho$ of $K_1$, let $\text{ind}_{K_1}^{K_2} \rho$ denote the compactly induced representation of $K_2$.
\begin{theorem}[{\cite[Theorem~15.1]{Yu}}]
The representation $\pi_i=\text{ind}_{K^i}^{G^i(F)} \rho_i$ of $G^i(F)$ is an irreducible supercuspidal representation of depth $r_i$ for $0\le i\le d$.
\end{theorem}
For the proof of this theorem, Yu uses \cite[Proposition~14.1]{Yu} and \cite[Theorem~14.2]{Yu}, which are pointed out in \cite{Fin1} to be false. However, Fintzen uses an alternative approach in \cite{Fin1} and proves \cite[Theorem~15.1]{Yu} without using \cite[Proposition~14.1]{Yu} or \cite[Theorem~14.2]{Yu}.

Next, we review Yu's construction of supercuspidal types.
We start with a datum  $(\overrightarrow{G}, y, \overrightarrow{r}, {^{\circ}\!\rho_{-1}}, \overrightarrow{\phi})$ satisfying {\bf D1}, {\bf D2}, {\bf D3}, {\bf D5} and $^\circ${\bf D4}:
\begin{description}
\item[$^\circ${\bf D4}]
${^\circ\!\rho_{-1}}$ is  an irreducible representation of $G^0(F)_{y}$ such that ${^{\circ}\!\rho_{-1}}\restriction_{G^0(F)_{y, 0}}$ is the inflation of a cuspidal representation of $G^0(F)_{y,0}/ G^0(F)_{y, 0+}$,
\end{description}
instead of {\bf D4}.
We then follow the above construction replacing $K^i$ with
\[^{\circ}\!K^i=G^0(F)_{y}G^1(F)_{y, r_0/2}\cdots G^i(F)_{y, {r_{i-1}}/2},\]
and construct an irreducible representation $^\circ\!\rho_i$ of $^\circ\!K^i$.
\begin{proposition}
\label{type}
Let $(\overrightarrow{G}, y, \overrightarrow{r}, \rho_{-1}, \overrightarrow{\phi})$ be a datum satisfying {\bf D1}, {\bf D2}, {\bf D3}, {\bf D4}, {\bf D5} and ${^{\circ}\!\rho_{-1}}$ be an irreducible representation of $G^0(F)_y$ which is contained in $\rho_{-1}\restriction_{G^0(F)_y}$. Then $(\overrightarrow{G}, y, \overrightarrow{r}, {^\circ\!\rho_{-1}}, \overrightarrow{\phi})$ satisfies {\bf D1}, {\bf D2}, {\bf D3}, $^\circ${\bf D4}, {\bf D5} and the representation $^\circ\!\rho_i$ constructed above is an $\mathfrak{s}_i$- type in the sense of \cite{BK2}, where $\mathfrak{s}_i$ is the inertial equivalence class of $[G^i, \pi_i]_{G^i}$.
\end{proposition}
\begin{proof}
Note that $^\circ\!K^i$ is the unique maximal compact subgroup of $K^i$, and $^\circ\!\rho_i$ is contained in $\rho_i\restriction_{^\circ\!K^i}$. Then this proposition follows from \cite[Proposition~5.4]{BK2}.
\end{proof}
Types obtained in this way are called \emph{tame supercuspidal types}.
In the following, we write $K^d, K^d_{+}, {^\circ\!K^{d}}, \rho_{d}, {^{\circ}\!\rho_{d}}, \pi_{d}$ simply by $K, K_{+}, {^\circ\!K}, \rho, {^\circ\!\rho}, \pi$, respectively.
\section{Hecke algebra isomorphism}
\label{sechecke}
Let $K$ be a compact open subgroup of $G(F)$ and $(\rho, W)$ be an irreducible representation of $K$. We define a Hecke algebra $\mathcal{H}(G(F), \rho)$ associated with $(K, \rho)$ as in \cite[Section~2]{BK2} and write $\check{\mathcal{H}}(G(F), \rho)$ for $\mathcal{H}(G(F), \check{\rho})$, where  $\check{\rho}$ is the contragradient of $\rho$. So, $\check{\mathcal{H}}(G(F), \rho)$ is the $\mathbb{C}$-vector space of compactly supported functions $\Phi: G(F)\to \text{End}_{\mathbb{C}}(W)$ satisfying
\[\Phi(k_1 g k_2)=\rho(k_1)\circ \Phi(g) \circ \rho(k_2), \ k_i\in K, g\in G(F),\]
and for $\Phi_1, \Phi_2\in \check{\mathcal{H}}(G(F), \rho)$, the product $\Phi_1*\Phi_2$ is defined by
\[\left(\Phi_1*\Phi_2\right)(x)=\int_{G(F)} \Phi_1(y)\circ\Phi_2(y^{-1}x) dy.\]
Here, we normalize the Haar measure of $G(F)$ so that $\text{vol}(K)=1$.
 If $\mathfrak{s}$ is an inertial equivalence class of a cuspidal pair and $(K, \rho)$ is an $\mathfrak{s}$-type, the Bernstein block associated with $\mathfrak{s}$ is equivalent to the category of $\mathcal{H}(G(F), \rho)$-modules \cite[Theorem~4.3]{BK2}.
We restrict our attention to Hecke algebras associated  with tame supercuspidal types, defined in Section \ref{secyu}.

Let $(\overrightarrow{G}, y, \overrightarrow{r}, \rho_{-1}, \overrightarrow{\phi})$ be a datum satisfying {\bf D1}, {\bf D2}, {\bf D3}, {\bf D4}, {\bf D5} and $^{\circ}\!\rho_{-1}$ be an irreducible representation of $G^0(F)_y$ which is contained in $\rho_{-1}\restriction_{G^0(F)_y}$. We construct a $[G, \pi]_G$-type $\left({^\circ\!K}, {^\circ\!\rho}\right)$ as in Section \ref{secyu}.

We define
\[\text{supp}\left(\check{\mathcal{H}}\left(G(F), {^\circ\!\rho}\right)\right)=\bigcup_{f\in \check{\mathcal{H}}\left(G(F), {^\circ\!\rho}\right)} \text{supp}(f),\]
where $\text{supp}(f)$ denotes the support of $f$.
We call it \emph{the support of $\check{\mathcal{H}}\left(G(F), {^\circ\!\rho}\right)$}. Note that
\[\text{supp}\left(\check{\mathcal{H}}\left(G(F), {^\circ\!\rho}\right)\right)=\{g\in G(F)\mid g \ \text{intertwines} \ ^\circ\!\rho\}.\]
\begin{proposition}
\label{support}
The support of $\check{\mathcal{H}}\left(G(F), {^\circ\!\rho}\right)$ is contained in ${^\circ\!K} G^0(F)_{[y]} {^\circ\!K}$. Moreover, an element $g\in G^0(F)_{[y]}$ intertwines $^\circ\!\rho$ if and only if $g$ intertwines $^\circ\!\rho_{-1}$.
\end{proposition}
\begin{remark}
\label{remsup}
By \cite[Remark~3.5]{Yu}, $G^0(F)_{[y]}$ normalizes $^\circ\!K$, so ${^\circ\!K} G^0(F)_{[y]} {^\circ\!K}=G^0(F)_{[y]} {^\circ\!K}$.
\end{remark}
Proposition~\ref{support} follows easily from \cite[Corollary~15.5]{Yu}. However \cite[Corollary~15.5]{Yu} relies on \cite[Proposition~14.1]{Yu} and \cite[Theorem~14.2]{Yu}, which are pointed out in \cite{Fin1} to be false. In the following, we prove Proposition~\ref{support} using an argument by Fintzen in \cite[Theorem~3.1]{Fin1}, which does not rely on \cite[Proposition~14.1]{Yu} or \cite[Theorem~14.2]{Yu}.

For the first claim, it is enough to show that if $g\in G(F)$ intertwines $^\circ\!\rho$, then $g\in {^\circ\!K}G^0(F)_{[y]} {^\circ\!K}$.
The first step is the following Lemma.
\begin{lemma}
If $g\in G(F)$ intertwines ${^\circ\!\rho}$, then $g \in {^\circ\!K}G^0(F) {^\circ\!K}$.
\end{lemma}
\begin{proof}
This follows from \cite[Proposition~4.1]{Yu} and \cite[Proposition~4.4]{Yu}. Note that \cite[Proposition~4.4]{Yu} is also true if we replace $\rho_i$ with $^\circ\!\rho$.
\end{proof}
Next, we prove the following Lemma.
\begin{lemma}
If $g\in G^0(F)$ intertwines $^\circ\!\rho$, then $g\in G^0(F)_{[y]}$.
\end{lemma}
\begin{proof}
Let $f$ be a nonzero element of $\text{Hom}_{^\circ\!K\cap ^g\!\left({^\circ\!K}\right)}\left(^g\!\left(^\circ\!\rho\right), {^\circ\!\rho}\right)$. We write $V_{f}$ for the image of $f$. By \cite[Proposition~4.4]{Yu}, $^{\circ}\!\rho\restriction_{K_{+}}$ is $\theta_{d}$-isotypic, where
\[\theta_{d}=\prod_{j=0}^{d} \hat{\phi}_{j}\restriction_{K_{+}}.\]
This implies that $G^0(F)_{y, 0+}\left(\subset K_{+}\right)$ acts on $^\circ\!\rho$ by $\theta_{d}$, and ${^gG^0(F)_{y, 0+}}$ acts on $^g\!\left(^\circ\!\rho\right)$ by $^g\theta_{d}$.
For $h\in G^0(F)_{y, 0}\cap {^g}G^0(F)_{y, 0+}$ and $0\le j\le d$,
\[{^g}\hat{\phi}_{j}(h)=\hat{\phi}_{j}(g^{-1}hg)=\phi_{j}(g^{-1}hg)=\phi_{j}(g)^{-1}\phi_{j}(h)\phi_{j}(g)=\phi_{j}(h)=\hat{\phi}_{j}(h).\]
Hence $G^0(F)_{y, 0}\cap {^g}G^0(F)_{y, 0+}$ acts on $^g\!\left(^\circ\!\rho\right)$ by $\theta_{d}$.
Therefore, if we let 
\[U'_{-1}=\left(G^0(F)_{y, 0}\cap {^g}G^0(F)_{y, 0+}\right)G^0(F)_{y, 0+},\]
$U'_{-1}$ acts on $V_f$ by $\theta_{d}$.

From the construction, $^\circ\!\rho$ is decomposed as
\[^\circ\!\rho=\bigotimes_{i=-1}^{d} V_{i},\]
where $V_{-1}$ is the inflation of $^\circ\!\rho_{-1}$ via the map
\[^\circ\!K=G^0(F)_{y}J^1\cdots J^d\to G^0(F)_y J^1\cdots J^d/J^1\cdots J^d\simeq G^0(F)_y/G^0(F)_{y, r_0},\]
$V_{i}$ is the inflation of $\phi'_{i}$ via the map
\[^\circ\!K={^\circ\!K}^{i+1}J^{i+2}\cdots J^{d}\to {^\circ\!K}^{i+1}J^{i+2}\cdots J^{d}/ J^{i+2}\cdots J^{d}={^\circ\!K}^{i+1}/G^{i+1}(F)_{y, r_{i+1}},\]
for $0\le i\le d-2$, $V_{d-1}=\phi'_{d-1}$, and $V_d=\phi_{d}$.

Let $T$ be a maximal torus of $G^0$ so that $T$ splits over a tamely ramified extension $E$ of $F$ and $y, g\!\cdot\!y\in \mathcal{A}(T, E)$. Such a torus exists by the fact that any two points of $\mathcal{B}(G^0, F)$ is contained in an apartment of $\mathcal{B}(G^0, F)$ and by the discussion in the beginning of \cite[Section~2]{Yu}.
We define
\[U_i=G(F) \cap \langle U_{\alpha}(E)_{y, r_i/2}\mid \alpha\in \Phi(G^{i+1}, T, E) \backslash \Phi(G^i, T, E), \alpha(y-g\!\cdot\!y)<0\rangle,\]
for $0\le i \le d-1$.
Since $U_i$ is contained in $J^{i+1}$, the action of $U_i$ on $V_{j}$ is trivial for $-1\le j\le i-1$. For, $i+1\le j$, $U_i$ is contained in ${^\circ\!K}^{j}$, which acts on $V_j$ by $\phi_{j}'\restriction_{{^\circ\!K}^j}=\phi_{j}\otimes \tilde{\phi}_j$. Here, the action of ${^\circ\!K}^{j}$ by $\tilde{\phi}_j$ factors through the map ${^\circ\!K}^j\to \text{Sp}\left(J^{j+1}/J^{j+1}_{+}\right)$ induced by the conjugation. By \cite[Proposition~6.4.44]{BT}, 
\[[J^{j+1}, U_{i}]\subset [J^{j+1}, G^{j}(F)_{y, 0+}]\subset J^{j+1}_{+}.\]
Therefore, the action of $U_i$ by $\tilde{\phi}_j$ is trivial, hence $U_i$ acts on $V_j$ by $\phi_{j}$.
For $j=i$, $U_i$ acts on $V_{i}$ by the Heisenberg representation $\tilde{\phi}_{i}$ of $J^{i+1}/J^{i+1}_{+}$.
Putting these arguments together, we see that the action of $U_i$ by $^\circ\!\rho$ is
\[\left(\otimes_{j=-1}^{i-1}\text{Id}_{V_{j}}\right)\otimes \tilde{\phi}_{i} \otimes \left(\otimes_{j=i+1}^{d} \phi_{j}\right).\]

On the other hand, for $\alpha\in \Phi(G^{i+1}, T, E) \backslash \Phi(G^i, T, E)$ which satisfies $\alpha(y-g\!\cdot\!y)<0$, we have
\begin{align*}
^{g^{-1}}U_{\alpha}(E)_{y, r_i/2}&= \,^{g^{-1}}U_{\alpha}(E)_{g\cdot y, (r_i/2) - \alpha(y-g\cdot y)}\\
&= U_{g^{-1}\alpha}(E)_{y, (r_i/2)-\alpha(y-g\cdot y)}\\
&\subset U_{g^{-1}\alpha}(E)_{y, (r_i/2)+}.
\end{align*}
Here, $g^{-1}\alpha$ denotes the character $t\mapsto \alpha(gtg^{-1})$ of $^{g^{-1}}T$, and $U_{g^{-1}\alpha}=U_{{^{g^{-1}}\!T}, g^{-1}\alpha}$ denotes the corresponding root subgroup.
Since $J^{i+1}_{+}$ is independent of the choice of a maximal torus, and $^{g^{-1}}T$ is a maximal torus of $G^0$ so that $^{g^{-1}}T$ splits over $E$ and $y\in \mathcal{A}(^{g^{-1}}T, E)$, $^{g^{-1}}U_i \subset J^{i+1}_{+}\subset K^{i+1}_{+}\subset K^d_{+}$. As $^\circ\!\rho\restriction_{K^d_{+}}$ is $\theta_{d}$-isotypic, $U_i$ acts on $V_f$ by $^g \theta_d\restriction_{U_i}$.

By the construction of $\hat{\phi}_{j}$ in \cite[Section~4]{Yu}, $\hat{\phi}_{j}$ is trivial on
\[G(F)\cap \langle U_{g^{-1}\alpha}(E)_{y, (r_j/2)+}\mid \alpha\in \Phi(G, T, E)\backslash \Phi(G^j, T, E) \rangle\]
for $0\le j\le d-1$.
Therefore, for $j\le i$, $^{g^{-1}}U_i$ is contained in $\text{Ker}(\hat{\phi}_{j})$. This implies that
\[^g \theta_d\restriction_{U_i}=\prod_{j={i+1}}^{d} {^g \hat{\phi}_j\restriction_{U_i}}=\prod_{j=i+1}^{d} {^g \phi_j\restriction_{U_i}}=\prod_{j=i+1}^{d} \phi_j\restriction_{U_i}.\]
Hence $U_i$ acts on $V_{f}$ by
\[\left(\otimes_{j=-1}^{i}\text{Id}_{V_{j}}\right)\otimes \left(\otimes_{j=i+1}^{d} \phi_{j}\right).\]
Comparing the action of $U_i$ by $^\circ\!\rho$ and the action of $U_i$ on $V_f$, we conclude that $V_f$ is contained in $V_{-1}\otimes (\otimes_{i=0}^{d-1} V_i^{U_i}) \otimes V_{d}$.

We study the subspace $V_i^{U_i}$ for $0\le i \le d-1$. Recall that $V_i$ is the space of the Weil representation of $\text{Sp}\left(J^{i+1}/J^{i+1}_{+}\right) \ltimes (J^{i+1}/J^{i+1}_{+})^\#$. We write $W^{i+1}=J^{i+1}/J^{i+1}_{+}$. We define the subspace $(W^{i+1})_{1}$ to be the image of $U_i$ in $W^{i+1}$, $(W^{i+1})_{2}$ to be the image of
\[G(F) \cap \langle U_{\alpha}(E)_{y, r_i/2}\mid \alpha\in \Phi(G^{i+1}, T, E) \backslash \Phi(G^i, T, E), \alpha(y-g\!\cdot\!y)=0\rangle\]
in $W^{i+1}$, and $(W^{i+1})_3$ to be the image of
\[G(F) \cap \langle U_{\alpha}(E)_{y, r_i/2}\mid \alpha\in \Phi(G^{i+1}, T, E) \backslash \Phi(G^i, T, E), \alpha(y-g\!\cdot\!y)>0\rangle\]
in $W^{i+1}$. Note that $(W^{i+1})_{k}$ is written by $V_k$ in \cite[Section~13]{Yu} for $1\le k\le 3$. By \cite[Lemma~13.6]{Yu}, $(W^{i+1})_1, (W^{i+1})_3$ are totally isotropic subspaces of the symplectic space $\left(W^{i+1}, \langle , \rangle_{i+1}\right)$ and
\[{(W^{i+1})_{1}}^{\perp}=(W^{i+1})_1\oplus (W^{i+1})_2, {(W^{i+1})_3}^{\perp} =(W^{i+1})_2\oplus (W^{i+1})_3.\]
Let $P_{i+1}$ be the maximal parabolic subgroup of $\text{Sp}(W^{i+1})$ that preserves $(W^{i+1})_{1}$. Then we obtain the natural map
\[\iota\colon P_{i+1}\to \text{Sp}\left({(W^{i+1})_{1}}^{\perp}/(W^{i+1})_{1}\right)\simeq\text{Sp}\left((W^{i+1})_2\right).\]
We write $(\tilde{\phi}_{i})_2$ for the Weil representation of $\text{Sp}\left((W^{i+1})_2\right)\ltimes((W^{i+1})_2)^{\#}$ associated with the central character $\hat{\phi}_{i}$. By \cite[Theorem~2.4 (b)]{Ge}, the restriction of $\tilde{\phi}_i$ from $\text{Sp}\left(W^{i+1}\right)\ltimes (W^{i+1})^{\#}$ to $P_{i+1}\ltimes (W^{i+1})^{\#}$ is given by
\[\text{ind}_{P_{i+1}\ltimes \left((W^{i+1})_1\oplus ((W^{i+1})_2)^{\#}\right)}^{P_{i+1} \ltimes (W^{i+1})^{\#}} (\tilde{\phi}_{i})_2 \otimes (\chi_{i+1}\ltimes 1).\]
Here, we regard $(\tilde{\phi}_{i})_2$ be a representation of $P_{i+1}\ltimes \left((W^{i+1})_1\oplus ((W^{i+1})_2)^{\#}\right)$ by defining the action of $(W^{i+1})_1$ to be trivial and defining the action of $P_{i+1}$ to be the composition of $\iota$ and $(\tilde{\phi}_i)_2$. The character $\chi_{i+1}$ of $P_{i+1}$ is $\chi^{E_{+}}$ of \cite[Lemma~2.3 (d)]{Ge}, which factors through the natural map $P_{i+1}\to \text{GL}((W^{i+1})_1)$.
Since $(W^{i+1})_3$ is a complete system of representatives for 
\[\left(P_{i+1} \ltimes (W^{i+1})^{\#}\right)/\left(P_{i+1}\ltimes \left((W^{i+1})_1\oplus ((W^{i+1})_2)^{\#}\right)\right),\]
as a representation of $P_{i+1}\ltimes \left((W^{i+1})_1\oplus ((W^{i+1})_2)^{\#}\right)$,
\[\text{ind}_{P_{i+1}\ltimes \left((W^{i+1})_1\oplus ((W^{i+1})_2)^{\#}\right)}^{P_{i+1} \ltimes (W^{i+1})^{\#}} (\tilde{\phi}_{i})_2 \otimes (\chi_{i+1}\ltimes 1) \simeq \bigoplus_{v_3\in (W^{i+1})_3} {^{v_3}\!(\tilde{\phi}_{i})_2 \otimes (\chi_{i+1}\ltimes 1)}.\]
Since $(W^{i+1})_1$ acts on $(\tilde{\phi}_i)_2$ trivially, $(W^{i+1})_1$ acts $^{v_3}\!(\tilde{\phi}_i)_2$ by
\[v_1 \mapsto \hat{\phi}_{i}({v_3}^{-1}v_1 v_3 {v_1}^{-1})=\langle{v_3}^{-1}, v_1 \rangle_{i+1}.\]
We note that ${(W^{i+1})_{3}}^{\perp}=(W^{i+1})_2\oplus (W^{i+1})_3$. Hence for every element $v_3\in (W^{i+1})_3$ there exists $v_1\in (W^{i+1})_1$ such that $\langle {v_3}^{-1}, v_1 \rangle_{i+1}\neq 0$.
Therefore,  
\[\left(\text{ind}_{P_{i+1}\ltimes (W^{i+1})_1\oplus ((W^{i+1})_2)^{\#}}^{P_{i+1} \ltimes (W^{i+1})^{\#}} (\tilde{\phi}_{i})_2 \otimes (\chi_{i+1}\ltimes 1)\right)^{\{1\}\ltimes ((W^{i+1})_1\times \{0\})}\simeq(\tilde{\phi}_{i})_2\otimes \chi_{i+1}\]
as a representation of $P_{i+1}$.

The image of $U_i$ via the special isomorphism constructed in \cite[Proposition~11.4]{Yu} is $(W^{i+1})_1\times \{0\}\subset (W^{i+1})^{\#}$. Therefore $P_{i+1}$ acts on $V_i^{U_i}$ by $(\tilde{\phi}_{i})_2\otimes \chi_{i+1}$.

We define
\[U_{-1}=G(F) \cap \langle U_{\alpha}(E)_{y, 0}\mid \alpha\in \Phi(G^{0}, T, E), \alpha(y-g\!\cdot\!y)<0\rangle.\]
Then, $U'_{-1}$ is contained in $U_{-1}G^0(F)_{y, 0+}$. By \cite[Proposition~6.4.44]{BT}, 
\[[J^{i+1}, G(F)_{y,0+}]\subset J^{i+1}_{+},\]
so the image of $G^0(F)_{y, 0+}$ in $\text{Sp}\left(W^{i+1}\right)$ is trivial.
Also, by \cite[Proposition~6.4.44]{BT},
\[[G^{i+1}(F)_{y, f_{(1,2)}}, U_{-1}]\subset G^{i+1}(F)_{y, f_1},\]
where $f{(1,2)}$ and $f_1$ are functions on $\Phi(G^{i+1}, T, E)\cup \{0\}$ defined by
\begin{align*}
f_{(1,2)}(\alpha)&=
\begin{cases}
{r_i} & \left(\alpha\in \Phi(G^i, T, E)\cup \{0\}\right)\\
\frac{r_i}{2} & \left(\alpha\in \Phi(G^{i+1}, T, E)\backslash \Phi(G^i, T, E), \alpha(y-g\!\cdot\!y)\le 0\right)\\
(\frac{r_i}{2})+ & \left(\alpha\in \Phi(G^{i+1}, T, E)\backslash \Phi(G^i, T, E), \alpha(y-g\!\cdot\!y)> 0\right),
\end{cases}\\
f_1(\alpha)&=
\begin{cases}
{r_i} & \left(\alpha\in \Phi(G^i, T, E)\cup \{0\}\right)\\
\frac{r_i}{2} & \left(\alpha\in \Phi(G^{i+1}, T, E)\backslash \Phi(G^i, T, E), \alpha(y-g\!\cdot\!y)< 0\right)\\
(\frac{r_i}{2})+ & \left(\alpha\in \Phi(G^{i+1}, T, E)\backslash \Phi(G^i, T, E), \alpha(y-g\!\cdot\!y)\ge 0\right)
\end{cases}
\end{align*}
and for $h=f_{(1,2)}, f_1$, 
\[G^{i+1}(F)_{y, h}=G(F)\cap \langle U_{\alpha}(E)_{y, h(\alpha)}\mid \alpha\in \Phi(G^{i+1}, T, E)\cup \{0\} \rangle.\]
Note that the image of $G^{i+1}(F)_{y,f_{(1,2)}}$ (resp.\,$G^{i+1}(F)_{y, f_1}$) in $W^{i+1}$ is $(W^{i+1})_1\oplus (W^{i+1})_2$ (resp.\, $(W^{i+1})_1$).
Therefore, the image of $U_{-1}$ in $\text{Sp}\left(W^{i+1}\right)$ is contained in $P_{i+1}$, and the image of $U_{-1}$ via the map
\[\iota\colon P_{i+1}\to \text{Sp}\left({(W^{i+1})_{1}}^{\perp}/(W^{i+1})_{1}\right)\simeq\text{Sp}\left((W^{i+1})_2\right)\]
is trivial. Moreover, since $U_{-1}$ is a prp-$p$ subgroup of $G^0(F)$, the image in $\text{GL}((W^{i+1})_1)$ under the natural map $P_{i+1}\to \text{GL}((W^{i+1})_1)$ is a $p$-group, hence contained in the commutator subgroup of $\text{GL}((W^{i+1})_1)$. Therefore, $\chi_{i+1}$ is trivial on the image of $U_{-1}$. These arguments imply that $\tilde{\phi}_i (U'_{-1})$ is trivial for $0\le i\le d-1$, so the action of $U'_{-1}$ on $(\otimes_{i=0}^{d-1} V_i^{U_i}) \otimes V_{d}$ is $\theta_d$-isotypic. Since $U'_{-1}$ acts on $V_f$ by $\theta_{d}$ and $V_f$ is contained in $V_{-1}\otimes (\otimes_{i=0}^{d-1} V_i^{U_i}) \otimes V_{d}$, this implies that $V_{-1}^{U'_{-1}}$ is nonzero. As $[y]$ is a vertex, if $g \not \in G^0(F)_{[y]}$, the image of $U'_{-1}$ in $G^0(F)_{y, 0}/G^0(F)_{y, 0+}$ is a unipotent radical of a proper parabolic subgroup of $G^0(F)_{y, 0}/G^0(F)_{y, 0+}$. This contradicts that ${^\circ\!\rho}_{-1}\restriction_{G^0(F)_{y, 0}}$ is the inflation of a cuspidal representation of $G^0(F)_{y,0}/ G^0(F)_{y, 0+}$.
Therefore we obtain that $g\in G^0(F)_{[y]}$.
\end{proof}
We prove the second claim of Proposition \ref{support} using an argument in \cite[Proposition~4.6]{Yu}. Since $G^0(F)_{[y]}$ is contained in $K$ and $V_{i}$ is a restriction of a representation of $K$ to $^\circ\!K$ for $0\le i\le d$, if $g\in G^0(F)_{[y]}$ intertwines ${^\circ\!\rho_{-1}}$, then $g$ intertwines $^\circ\!\rho$. We prove that the converse is true. Assume that $g\in G^0(F)_{[y]}$ intertwines ${^\circ\!\rho}$. Since $V_d$ is a restriction of a character of $G(F)$, $g$ also intertwines $\otimes_{j=-1}^{d-1} V_i$. If $d=0$ it implies that $g$ intertwines ${^\circ\!\rho_{-1}}$. Suppose $d\ge 1$. We prove that $g$ intertwines $\otimes_{j=-1}^{d-2} V_i$.
Let $f$ be a nonzero element of $\text{Hom}_{^\circ\!K}\left(^g\!\left(^\circ\!\rho\right), {^\circ\!\rho}\right)$. We write $f=\sum_{j} f'_j\otimes f''_j$, where $f'_j\in \text{End}_{\mathbb{C}}(\otimes_{i=-1}^{d-2} V_i)$ and $f''_j\in \text{End}_{\mathbb{C}}(V_{d-1})$. We may assume that $\{f'_j\}$ is a linearly independent set.

Since the action of $J^{d}$ on $\otimes_{i=-1}^{d-2} V_i$ is trivial, for $x\in J^{d}$ we obtain
\[\sum_{j} f'_j\otimes \left(f''_j\circ (^g \phi'_{d-1})(x)\right)=\sum_{j} f'_j \otimes \left(\phi'_{d-1}(x)\circ f''_j\right).\] 
The linearly independence of $\{f_j\}$ implies that $f''_j\in \text{Hom}_{J^d}(^g\phi'_{d-1}, \phi'_{d-1})$. By \cite[Proposition~12.3]{Yu}, $\text{Hom}_{J^d}(^g\phi'_{d-1}, \phi'_{d-1})$ is $1$-dimensional, so we may assume that there is only one $j$. We write $f=f'\otimes f''$, where $f'\in \text{End}_{\mathbb{C}}(\otimes_{i=-1}^{d-2} V_i)$ and $f''\in \text{Hom}_{J^d}(^g\phi'_{d-1}, \phi'_{d-1})$. Since $g\in G^0(F)_{[y]}\subset K$, $\text{Hom}_{K}(^g\phi'_{d-1}, \phi'_{d-1})=\text{End}_K(\phi'_{d-1})$ is $1$-dimensional, and it is a subspace of $\text{Hom}_{J^d}(^g\phi'_{d-1}, \phi'_{d-1})$, which is also $1$-dimensional. Therefore $\text{Hom}_{J^d}(^g\phi'_{d-1}, \phi'_{d-1})=\text{Hom}_{K}(^g\phi'_{d-1}, \phi'_{d-1})$, and $f''$ is $K$-equivariant. This implies that $f'$ is a nonzero element in $\text{Hom}_{^\circ\!K}(^g(\otimes_{j=-1}^{d-2} V_i), \otimes_{j=-1}^{d-2} V_i)$, and $g$ intertwines $\otimes_{j=-1}^{d-2} V_i$. Then, an inducting argument implies that $g$ intertwines $^\circ\!\rho_{-1}$.

We now prove \cite[Conjecture~0.2]{Yu}.
\begin{theorem}
There is a support-preserving algebra isomorphism
\[\check{\mathcal{H}}\left(G(F), {^\circ\!\rho}\right)\simeq \check{\mathcal{H}}\left(G^0(F), {^\circ\!\rho_{-1}}\right).\]
\end{theorem}
Here, we say that an isomorphism $\eta:\check{\mathcal{H}}\left(G(F), {^\circ\!\rho}\right)\to \check{\mathcal{H}}\left(G^0(F), {^\circ\!\rho_{-1}}\right)$ is support-preserving if for every $f\in \check{\mathcal{H}}\left(G(F), {^\circ\!\rho}\right)$, $\text{supp}(f)={^\circ\!K} \text{supp}(\eta(f)) {^\circ\!K}$.
\begin{proof}
We set
\[G^0_{^\circ\!\rho}=\{g\in G^0(F)_{[y]}\mid \text{$g$ intertwines ${^\circ\!\rho}$}\}=\{g\in G^0(F)_{[y]}\mid \text{$g$ intertwines ${^\circ\!\rho_{-1}}$}\}.\]
The second equation follows from Proposition \ref{support}.
By Proposition \ref{support} and Remark \ref{remsup}, the support of $\check{\mathcal{H}}\left(G(F), {^\circ\!\rho}\right)$ is $G^0_{^\circ\!\rho} {^\circ\!K}$, and the support of $\check{\mathcal{H}}\left(G^0(F), {^\circ\!\rho_{-1}}\right)$ is $G^0_{^\circ\!\rho} {^\circ\!K^0}$. Fix a complete system of representatives $\{g_i\}_{i}\subset G^0_{^\circ\!\rho}$ for
\[G^0_{^\circ\!\rho} {^\circ\!K}/{^\circ\!K}=G^0_{^\circ\!\rho} {^\circ\!K^0}/{^\circ\!K^0}=G^0_{^\circ\!\rho}/\left(G^0_{^\circ\!\rho}\cap{^\circ\!K^0}\right).\]
For each $g_i$, $\text{Hom}_{^\circ\!K^0}(^{g_i}({^\circ\!\rho_{-1}}), {^\circ\!\rho_{-1}})$ is $1$-dimensional. We fix a basis $(T_{g_i})_{-1}$ of this space. We define
\[T_{g_i}=(T_{g_{i}})_{-1}\otimes \phi'_0(g_i)\otimes \ldots \otimes \phi'_{d}(g_i).\]
The element $T_{g_i}$ is a basis of the $1$-dimensional vector space $\text{Hom}_{^\circ\!K}(^{g_i}\!({^\circ\!\rho}), {^\circ\!\rho})$.
We define $f_{g_i}\in \check{\mathcal{H}}\left(G(F), {^\circ\!\rho}\right)$ and $(f_{g_i})_{-1}\in \check{\mathcal{H}}\left(G^0(F), {^\circ\!\rho_{-1}}\right)$ by
\begin{align*}
f_{g_i}(x)&=
\begin{cases}
T_{g_i}\circ {^\circ\!\rho(k)} & (x=g_ik, k\in {^\circ\!K})\\
0 & (\text{otherwise}),
\end{cases}\\
(f_{g_i})_{-1}(x)&=
\begin{cases}
(T_{g_i})_{-1}\circ {{^\circ\!\rho_{-1}}(k)} & (x=g_ik, k\in {^\circ\!K^0})\\
0 & (\text{otherwise}).
\end{cases}
\end{align*}
Then as vector spaces, 
\[\check{\mathcal{H}}\left(G(F), {^\circ\!\rho}\right)=\bigoplus_{i} \mathbb{C}f_{g_i}, \check{\mathcal{H}}\left(G^0(F), {^\circ\!\rho_{-1}}\right)=\bigoplus_{i} \mathbb{C}(f_{g_i})_{-1}\]
and $f_{g_i}\mapsto (f_{g_i})_{-1}$ gives a support preserving vector space isomorphism $\check{\mathcal{H}}\left(G(F), {^\circ\!\rho}\right) \to \check{\mathcal{H}}\left(G^0(F), {^\circ\!\rho_{-1}}\right)$. We will show that it is an algebra isomorphism. Let $g_{i_1}, g_{i_2}$ be representatives, and take $g_{i_3}$ so that $g_{i_1} g_{i_2}\in g_{i_3} {^\circ\!K^{0}}$. We simply write $g_1, g_2,g_3$ for $g_{i_1}, g_{i_2}, g_{i_3}$, respectively. For $x\in G$, we have
\begin{align*}
(f_{g_1}*f_{g_2})(x)&=\int_{G} f_{g_1}(y) \circ f_{g_2}(y^{-1}x) dy\\
&= \int_{{^\circ\!K}} T_{g_1}\circ {^\circ\!\rho}(k) \circ f_{g_2}(k^{-1}{g_1}^{-1}x) dk\\
&=
\begin{cases}
T_{g_1} \circ T_{g_2} \circ {^\circ\!\rho} ((g_1 g_2)^{-1} g_3 k')& (x=g_{3}k', k'\in {^\circ\!K})\\
0 & (\text{otherwise})
\end{cases}\\
&=c\cdot f_{g_3}(x),
\end{align*}
where $c$ satisfies
\[c\cdot T_{g_3}=T_{g_1}\circ T_{g_2} \circ {^\circ\!\rho}((g_1g_2)^{-1}g_3).\]
By a similar calculation, we obtain
\[(f_{g_1})_{-1} *(f_{g_2})_{-1}=c_{-1}\cdot (f_{g_3})_{-1},\]
where $c_{-1}$ satisfies
\[c_{-1}\cdot (T_{g_3})_{-1}=(T_{g_1})_{-1}\circ (T_{g_2})_{-1} \circ {^\circ\!\rho_{-1}}((g_1g_2)^{-1}g_3).\]
By the definition of $T_{g_i}$,
\begin{align*}
&T_{g_1}\circ T_{g_2} \circ {^\circ\!\rho}((g_1g_2)^{-1}g_3)\\
&=\left((T_{g_1})_{-1}\circ (T_{g_2})_{-1} \circ {^\circ\!\rho_{-1}}((g_1g_2)^{-1}g_3)\right)\otimes \left(\bigotimes_{j=0}^{d} \phi'_j(g_1) \circ \phi'_j(g_2) \circ \phi'_j((g_1 g_2)^{-1}g_3)\right)\\
&=\left((T_{g_1})_{-1}\circ (T_{g_2})_{-1} \circ {^\circ\!\rho_{-1}}((g_1g_2)^{-1}g_3)\right)\otimes \left(\bigotimes_{j=0}^{d} \phi'_j(g_3)\right)\\
&= c_{-1}\cdot (T_{g_3})_{-1}\otimes \left(\bigotimes_{j=0}^{d} \phi'_j(g_3)\right)\\
&= c_{-1}\cdot T_{g_3}.
\end{align*}
This implies that $c=c_{-1}$ and the isomorphism constructed above is an algebra isomorphism.
\end{proof}
\section{Hecke algebras for regular supercuspidal types}
\label{secreg}
Firstly, we review the definition and the construction of regular supercuspidal representations. In this section, we assume that $p$ is odd, is not a bad prime for $G$, and dose not divide the order of the fundamental group of $G_{\text{der}}$. These assumptions are needed for the existence of a Howe factorization.

Let $\left(S, \theta \right)$ be a regular tame elliptic pair, i.\,e., $S$ is a tame elliptic maximal torus of $G$, and $\theta$ is a character of $S(F)$ which satisfy the conditions in \cite[Definition~3.7.5]{Kal}. Let $[y]$ be the point of reduced building of $G$ over $F$ which is associated to $S$ in the sense of the paragraph above \cite[Lemma~3.4.3]{Kal} and chose $y\in \mathcal{B}(G, F)$ such that the projection of $y$ on the reduced building is $[y]$.

From $\left(S, \theta \right)$, we define a sequence of twisted Levi subgroups
\[\overrightarrow{G}=\left(S=G^{-1}\subset G^0\subsetneq \ldots \subsetneq G^d=G\right)\]
in $G$ and a sequence of real numbers $\overrightarrow{r}=\left(0=r_{-1}, r_0, \ldots , r_d\right)$ as in \cite[3.6]{Kal}.

A Howe factorization of $\left(S, \theta\right)$ is a sequences of characters $\overrightarrow{\phi}=\left(\phi_{-1}, \ldots, \phi_d\right)$, where $\phi_i$ is a character of $G^i(F)$ satisfying
\[\theta=\prod_{i=-1}^{d} \phi_i\restriction_{S(F)}\]
and some additional technical conditions (see \cite[Definition~3.6.2]{Kal}).
By \cite[Proposition~3.6.7]{Kal}, $\left(S, \theta\right)$ has a Howe factorization.
We take a Howe factorization $\overrightarrow{\phi}$.
Using the pair $(S, \phi_{-1})$, we define an irreducible representation $\rho_{-1}$ of $G^0(F)_{[y]}$ as follows.

Let $\mathsf{G}_{y}^{\circ}$ be the reductive quotient of the special fiber of the connected parahoric group scheme of $G^0$ associated to $y$ and $\mathsf{S}^{\circ}$ be the reductive quotient of the special fiber of the connected N\'eron model of $S$. Then $\mathsf{S}^\circ\subset \mathsf{G}_{y}^\circ$ is an elliptic maximal torus. The restriction of $\phi_{-1}$ to $S(F)_{0}$ factors through a character $\bar{\phi}_{-1}: \mathsf{S}^\circ(k_F)\to \mathbb{C}^{\times}$.

Let $\kappa_{(S, \bar{\phi}_{-1})}=\pm R_{\mathsf{S}^\circ, \bar{\phi}_{-1}}$ be the irreducible cuspidal representation of $\mathsf{G}_{y}^\circ(k_F)$ arising from the Deligne--Lusztig construction applied to $\mathsf{S}^\circ$ and $\bar{\phi}_{-1}$ \cite[Section~1]{DL}. We identify it with its inflation to $G^0(F)_{y, 0}$. We can extend $\kappa_{(S, \bar{\phi}_{-1})}$ to a representation $\kappa_{(S, \phi_{-1})}$ of $S(F) G^0(F)_{y, 0}$ \cite[3.4.4]{Kal}.
Now we define $\rho_{-1}=\text{ind}_{S(F) G^0(F)_{y, 0}}^{G^0(F)_{[y]}} \kappa_{(S, \phi_{-1})}$.
\begin{proposition}
Let $\left(S, \theta \right)$ be a regular tame elliptic pair and $G^i, r_i, \phi_i, \rho_{-1}$ be as above. Then, $\left((G^i)_{i=0}^{d}, y, (r_{i})_{i=0}^{d}, \rho_{-1}, (\phi_i)_{i=0}^{d}\right)$ satisfies {\bf D1}, {\bf D2}, {\bf D3}, {\bf D4}, {\bf D5}.
\end{proposition}
\begin{proof}
This is a part of \cite[Proposition~3.7.8]{Kal}.
\end{proof}
Using this datum, we construct an irreducible supercuspidal representation $\pi_{\left(S, \theta\right)}$ of $G(F)$, which only depends on $(S, \theta)$. An irreducible supercuspidal representation of $G(F)$ obtained in this way is called \emph{regular}.

Next, we define \emph{regular supercuspidal types}.
Let $\left(S, \theta \right)$ be a regular tame elliptic pair and define $G^i, r_i, \phi_i, \kappa_{(S, \phi_{-1})}$ as above. We define
\[^\circ\!\rho_{-1}=\text{ind}_{G^0(F)_{y}\cap S(F) G^0(F)_{y, 0}}^{G^0(F)_y} {^\circ\!\kappa_{(S, \phi_{-1})}},\]
where $^\circ\!\kappa_{(S, \phi_{-1})}$ denotes the restriction of $\kappa_{(S,\phi_{-1})}$ to $G^0(F)_{y}\cap S(F) G^0(F)_{y, 0}$.
\begin{proposition}
\label{irr}
The representation $^\circ\!\rho_{-1}$ is an irreducible representation of $G^0(F)_y$.
\end{proposition}
\begin{proof}
This is essentially the same as \cite[Proposition~3.4.20]{Kal}, which is the same as the one for \cite[Lemma~4.5.1]{DR}.

Since $G^0(F)_{y}\cap S(F) G^0(F)_{y, 0}$ contains $G^0(F)_{y, 0}$ and $\kappa_{(S, \phi_{-1})}\restriction_{G^0(F)_{y,0}}=\kappa_{(S, \bar{\phi}_{-1})}$ is irreducible, $^\circ\!\kappa_{(S, \phi_{-1})}$ is also irreducible. Therefore it is enough to show that if $g\in G^0(F)_y$ intertwines $^\circ\!\kappa_{(S, \phi_{-1})}$, then $g\in G^0(F)_{y}\cap S(F) G^0(F)_{y, 0}$. Suppose $g\in G^0(F)_y$ intertwines $^\circ\!\kappa_{(S, \phi_{-1})}$, then $g$ also intertwines $\kappa_{(S, \bar{\phi}_{-1})}$. Hence by \cite[Theorem~6.8]{DL}, there is $h\in G^0(F)_{y,0}$ so that $\text{Ad}(hg)\left(\mathsf{S}^{\circ}, \bar{\phi}_{-1}\right)=\left(\mathsf{S}^{\circ}, \bar{\phi}_{-1}\right)$. By \cite[Lemma~3.4.5]{Kal}, there is $l\in G^0(F)_{y, 0+}$ so that $\text{Ad}(lhg)\left(S, \phi_{-1}\restriction_{S(F)_0}\right)=\left(S, \phi_{-1}\restriction_{S(F)_0}\right)$. Thus $lhg\in S(F)$ by the regularity of $\theta$, and $g\in G^0(F)_y\cap S(F) G^0(F)_{y,0}$.
\end{proof}
This proposition implies that $\left((G^i)_{i=0}^{d}, y, (r_{i})_{i=0}^{d}, {^\circ\!\rho_{-1}}, (\phi_i)_{i=0}^{d}\right)$ satisfies {\bf D1}, {\bf D2}, {\bf D3}, {$^\circ${\bf D4}}, {\bf D5}. We construct a $[G, \pi_{\left(S, \theta\right)}]_G$-type $\left({^\circ\!K}, {^\circ\!\rho}\right)$ from this datum. We call types obtained in this way \emph{regular supercuspidal types}.
\begin{remark}
Let $\left(S, \theta\right)$ be a regular tame elliptic pair. Let $\rho_{-1}$ be the representation of $G^0(F)_{[y]}$ and $^\circ\!\rho_{-1}$ be the representation of $G^0(F)_{y}$ defined as above. If $\rho'$ is an irreducible representation of $G^0(F)_y$ which is contained in $\rho_{-1}\restriction_{G^0(F)_y}$ then there is $g\in G^0(F)_{[y]}$ so that $\rho'$ is equivalent to $^g\!\left(^\circ\!\rho_{-1}\right)$. Now,
\[\rho'\simeq ^g\!\left(^\circ\!\rho_{-1}\right)\simeq \text{ind}_{G^0(F)_{y}\cap gS(F)g^{-1} G^0(F)_{y, 0}}^{G^0(F)_y} {^\circ\!\kappa_{(gSg^{-1}, ^{g}\!\phi)}}\]
and $\left(gSg^{-1}, ^{g}\!\phi\right)$ is a regular tame elliptic pair. Therefore, the $[G, \pi_{\left(S, \theta\right)}]_G$-type constructed by the datum $\left((G^i)_{i=0}^{d}, y, (r_{i})_{i=0}^{d}, \rho', (\phi_i)_{i=0}^{d}\right)$ is regular supercuspidal.
\end{remark}
We determine the structure of Hecke algebras associated with regular supercuspidal types. So, let $\left(S, \theta\right)$ be a regular tame elliptic pair and $\left({^{\circ}\!K}, {^{\circ}\!\rho}\right)$ be as above. We consider the  Hecke algebra $\check{\mathcal{H}}\left(G, {^\circ\!\rho}\right)$ associated with the type $\left({^\circ\!K}, {^\circ\!\rho}\right)$.
\begin{proposition}
The support of $\check{\mathcal{H}}\left(G, {^\circ\!\rho}\right)$ is $S(F) {^{\circ}\!K}$.
\end{proposition}
\begin{proof}
By Proposition \ref{support}, it is enough to show that $g\in G^0(F)_{[y]}$ intertwines $^\circ\!\rho_{-1}$ if and only if $g\in S(F)G^0(F)_{y}$. If $g\in S(F)G^0(F)_{y}$, it is obvious that $g$ intertwines $^\circ\!\rho_{-1}$. Conversely, suppose that $g\in G^0(F)_{[y]}$ intertwines $^\circ\!\rho_{-1}$. By the construction of $^\circ\!\rho_{-1}$, there exists $g_{0}\in G^0(F)_{y}$ so that $gg_{0}$ intertwines $^\circ\!\kappa_{(S, \phi_{-1})}$. Then, as in the proof of Proposition \ref{irr}, we conclude that $gg_{0}\in S(F)G^0(F)_{y, 0}$, and so $g\in S(F)G^0(F)_{y}$. 
\end{proof}
We define $S(F)_{b}={^\circ\!K}\cap S(F)=G^0(F)_{y} \cap S(F)$, which is the unique maximal compact subgroup of $S(F)$.
\begin{corollary}
The algebra $\check{\mathcal{H}}\left(G, {^\circ\!\rho}\right)$ is isomorphic to the group algebra $\mathbb{C}[S(F)/S(F)_{b}]$ of $S(F)/{S(F)_{b}}$.
\end{corollary}
\begin{proof}
Since ${^\circ\!\rho_{-1}}$ is the restriction of $\text{ind}_{S(F) G^0(F)_{y,0}}^{S(F) G^0(F)_{y}} \kappa_{(S, \phi_{-1})}$ to $G^0(F)_{y}$, we can extend ${^\circ\!\rho_{-1}}$ to $S(F) G^0(F)_{y}$. Therefore, we can extend ${^\circ\!\rho}$ to $S(F) {^\circ\!K}$. For $[g]=g {^\circ\!K}\in S(F) {^\circ\!K}/{^\circ\!K}$, we define $f_{[g]} \in \check{\mathcal{H}}\left(G, {^\circ\!\rho}\right)$ by
\begin{align*}
f_{[g]}(x)&=
\begin{cases}
{^\circ\!\rho(x)} & (x\in [g])\\
0 & (\text{otherwise}).
\end{cases}
\end{align*}
Then, as a vector space
\[\check{\mathcal{H}}\left(G(F), {^\circ\!\rho}\right)=\bigoplus_{[g]\in S(F) {^\circ\!K}/{^\circ\!K}} \mathbb{C}f_{[g]}\]
and for $[g], [h]\in S(F) {^\circ\!K}/{^\circ\!K}$, $f_{[g]}*f_{[h]}=f_{[gh]}$. Hence $\check{\mathcal{H}}\left(G(F), {^\circ\!\rho}\right)$ is isomorphic to the group algebra of $S(F) {^\circ\!K}/{^\circ\!K}\simeq S(F)/S(F)_{b}$.
\end{proof}

\end{document}